\documentclass[12pt, notitlepage]{article}

\usepackage{mystyle}

\usepackage{graphicx}
\usepackage[centertags]{amsmath}
\usepackage{newlfont}
\usepackage{hyperref}

\usepackage[page,title,titletoc,header]{appendix}
\usepackage{relsize}
\usepackage{stmaryrd}

\newcommand{\bB}{\mathbb{B}}

\newcommand{\acts}{\curvearrowright}
\newcommand{\symd}{\triangle}
\newcommand{\Stab}{\mathrm{Stab}}

\newcommand{\dom}{\mathrm{dom}}

\newcommand{\sH}{\mathrm{H}}

\newcommand{\pmp}{p{$.$}m{$.$}p{$.$}}
\renewcommand{\:}{\,:\,}
\newcommand{\res}{\restriction}

\begin{document}
\pagestyle{headings}
\title{Borel structurability on the 2-shift of a countable group}
\date{}
\author{Brandon Seward\footnote{email:b.m.seward@gmail.com}\\
University of Michigan\footnote{Department of Mathematics, University of Michigan, 530 Church Street, Ann Arbor, MI 48109, U.S.A.},\\
$\ $\\
Robin D. Tucker-Drob\footnote{email:rtuckerd@gmail.com}\\
Rutgers University\footnote{Department of Mathematics, Rutgers University - Hill Center for the Mathematical Sciences, 110 Frelinghuysen Rd., Piscataway, NJ 08854-8019, U.S.A.}{\let\thefootnote\relax\footnote{{\bf{Keywords:}} Bernoulli shift, Borel reducibility, Borel structurability, Borel combinatorics, factor map, entropy}}}
\begin{abstract}
We show that for any infinite countable group $G$ and for any free Borel action $G\cc X$ there exists an equivariant class-bijective Borel map from $X$ to the free part $\text{Free}(2^G)$ of the $2$-shift $G\cc 2^G$. This implies that any Borel structurability which holds for the equivalence relation generated by $G\cc \text{Free}(2^G)$ must hold \emph{a fortiori} for all equivalence relations coming from free Borel actions of $G$. A related consequence is that the Borel chromatic number of $\text{Free}(2^G)$ is the maximum among Borel chromatic numbers of free actions of $G$. This answers a question of Marks. Our construction is flexible and, using an appropriate notion of genericity, we are able to show that in fact the generic $G$-equivariant map to $2^G$ lands in the free part. As a corollary we obtain that for every $\epsilon >0$, every free {\pmp} action of $G$ has a free factor which admits a $2$-piece generating partition with Shannon entropy less than $\epsilon$. This generalizes a result of Danilenko and Park.
\end{abstract}
\maketitle
\setcounter{tocdepth}{3}
\tableofcontents
%

\section{Introduction}

Let $G$ be a countably infinite discrete group. For a Polish space $K$, we equip $K^G = \prod_{g \in G} K$ with the product topology and we let $G$ act on $K^G$ via the left shift action: $(g \cdot w)(h) = w(g^{-1} h)$ for $g, h \in G$ and $w \in K^G$. We call $K^G$ the {\bf $K$-shift}. For $W \subseteq K^G$ we write $\ol{W}$ for the closure of $W$. The {\bf free part of $K^G$}, denoted $\mathrm{Free}(K^G)$, is the set of points having trivial stabilizer:
$$\mathrm{Free}(K^G) = \{w \in K^G \: \forall g \in G \ g \neq 1_G \Longrightarrow g \cdot w \neq w\}.$$
We mention that, unless $|K| = 1$, the set $\mathrm{Free}(K^G)$ is not closed in $K^G$. We will work almost exclusively with the $2$-shift $2^G$, where we use the convention that $2 = \{0, 1\}$.

Let $G\cc X$ be a Borel action of $G$ on a standard Borel space $X$. Our starting point is the well-known bijective correspondence
\begin{align*}
\big\{ \text{Borel subsets of }X\big\} &\longleftrightarrow \big\{ G\text{-equivariant Borel maps from }X\text{ into }2^G \big\} ,
\end{align*}
which sends a Borel subset $A\subseteq X$ to the map $f_A : X\ra 2^G$ given by $f_A(x)(g)= 1_{g\cdot A}(x)$, and whose inverse sends a $G$-equivariant Borel map $f:X\ra 2^G$ to the set $A_f = \{ x\in X\csuchthat f(x)(1_G) = 1 \}$. Since the map $f_A$ encodes information not only about the set $A$, but also about each of its infinitely many translates $\{ g\cdot A \} _{g\in G}$, it is not surprising that properties of $f_A$ can depend very subtly on $A$. In this article, we provide a flexible construction, based on a construction of Gao, Jackson, and Seward \cite{GJS12}, of subsets $A\subseteq X$ that yield $G$-equivariant Borel maps into the free part $\mathrm{Free}(2^G)$ of $2^G$, under the assumption that the action $G\cc X$ is free. It is easy to see that freeness of $G\cc X$ is a necessary condition for the existence of such maps. Our main result moreover shows that, when the action $G\cc X$ is free, not only do such maps exist, but they are abundant.

In what follows, we call a subset $M\subseteq X$ {\bf syndetic} if $X= F\cdot M$ for some finite $F\subseteq G$. Also, if $\mu$ is a Borel probability measure on $X$, then recall that the measure algebra $\mathrm{MALG}_\mu$ is the collection of Borel subsets of $X$ modulo $\mu$-null sets. It is a Polish space under the metric $d([A]_\mu, [B]_\mu) = \mu(A \symd B)$, where $[A]_\mu$ denotes the equivalence class of $A$ in $\mathrm{MALG}_\mu$ and $\symd$ denotes symmetric difference.

\begin{theorem}\label{thm:main}
Let $G \acts X$ be a free Borel action of $G$ on a standard Borel space $X$. Then there exists a $G$-equivariant Borel map $f : X\ra 2^G$ with $\ol{f(X)}\subseteq \mathrm{Free}(2^G)$. Furthermore:
\begin{enumerate}
\item Suppose that $Y\subseteq X$ is a Borel set such that $X\setminus Y$ is syndetic, and $\phi : Y\ra 2$ is a Borel function. Then there exists a $G$-equivariant Borel map $f : X\ra 2^G$ with $\ol{f(X)}\subseteq \mathrm{Free}(2^G)$ and $f(y)(1_G) = \phi (y)$ for all $y\in Y$.

\item Let $Y$ and $\phi$ be as in part (1). Then there exists a family $\{ f_w \} _{w\in 2^\N}$ of maps each satisfying the conclusion of part (1), and with the further property that
    \[
    \ol{f_w(X)}\cap \ol{f_z(X)} = \emptyset
    \]
    for all distinct $w,z\in 2^\N$. In addition, the map $(w,x) \mapsto f_w(x)$ is Borel, and for each fixed $x\in X$ the map $w \mapsto f_w(x)$ is continuous.

\item For any $G$-quasi-invariant Borel probability measure $\mu$ on $X$, the set
    \[
    \{ [A]_\mu \csuchthat A\subseteq X\text{ is Borel and }f_A(X)\subseteq \mathrm{Free}(2^G)\}
    \]
    is dense $G_\delta$ in $\mathrm{MALG}_{\mu}$.
\end{enumerate}
\end{theorem}

In general the maps $f : X \rightarrow 2^G$ provided by the above theorem will not be injective. For example, if $G$ is amenable (or more generally sofic) and $G \acts X$ admits an invariant Borel probability measure $\mu$, then there cannot exist an equivariant injection into $2^G$ if the entropy of $G \acts (X, \mu)$ is greater than $\log(2)$. We mention, however, that a long standing open problem due to Weiss asks whether there is an equivariant injection $f : X \rightarrow k^G$ for some $k \in \N$ whenever $G \acts X$ does not admit any invariant Borel probability measure, see \cite[p. 324]{We89} and \cite[Problem 5.7]{JKL02}. Tserunyan \cite{Ts12} has shown that such an injection does exist whenever $G\cc X$ admits a $\sigma$-compact realization, although in general the problem remains open even in the case $G=\Z$.

Theorem \ref{thm:main} has a number of applications. For example, it implies that if the equivalence relation generated by $G \cc \mathrm{Free}(2^G)$ is treeable, then all equivalence relations induced by free Borel actions of $G$ are treeable. It also implies that $G \cc \mathrm{Free}(2^G)$ has maximal Borel chromatic number among all free Borel actions of $G$, and that every probability measure preserving action of $G$ has free factors which are arbitrarily small in the sense of Shannon entropy. We discuss these applications at length in \S\ref{sec:consequences} below. Then statements (1) and (2) of Theorem \ref{thm:main} are proved in \S\ref{sec:construction} via an inductive construction which is based on methods from \cite[Chapter 10]{GJS12}. Finally, statement (3) is deduced from (1) in \S\ref{sec:generic}.


\section{Consequences of Theorem 1.1}\label{sec:consequences}

\subsection{Borel structurability}
Let $E$ and $F$ be countable Borel equivalence relations on the standard Borel spaces $X$ and $Y$ respectively. A {\bf homomorphism} from $E$ to $F$ is a map $f:X\ra Y$ which takes $E$-equivalent points to $F$-equivalent points. Such a homomorphism is called {\bf class-bijective} if for each $x\in X$, the restriction of $f$ to the $E$-class $[x]_E$ is a bijection onto the $F$-class $[f(x)]_F$. A class-bijective homomorphism $f:X\ra Y$ from $E$ to $F$ may be viewed as a structurability reduction from $E$ to $F$; any structuring on the $F$-classes can be pulled back, via the map $f$, to obtain a structuring \emph{of the same isomorphism type} on the $E$-classes.

More precisely, let $L = ( R_i ) _{i\in I}$ be a countable relational language, where $R_i$ has arity $n_i$, and let $\mc{K}$ be a class of countable $L$-structures that is closed under isomorphism. The equivalence relation $E$ is said to be {\bf Borel $\mc{K}$-structurable} if there exists a collection $(Q_i)_{i\in I}$ of Borel sets with $Q_i \subseteq \{ (x_0,x_1,\dots , x_{n_i-1}) \in X^{n_i}\csuchthat x_0Ex_1\cdots Ex_{n_i-1} \}$ for each $i\in I$, such that for every $x\in X$, the $L$-structure $\langle [x]_E, (Q_i\res [x]_E )_{i\in I}\rangle$ is in $\mc{K}$. The collection $(Q_i )_{i \in I}$ is called a {\bf Borel $\mc{K}$-structuring} of $E$. For example, if $\mc{K}$ consists of the class of countable trees, then the Borel $\mc{K}$-structurable equivalence relations are precisely the treeable equivalence relations. The notion of Borel structurability was introduced in \cite[\S 2.5]{JKL02}. See \cite{Ma13b} and \cite{Ke14} for recent work in this area.

It is an easy exercise to see that Borel structurings can be pulled back through class-bijective homomorphisms, yielding the following.

\begin{proposition}\label{prop:Kstruct}
Suppose that there exists a class-bijective Borel homomorphism $f:X\ra Y$ from $E$ to $F$. If $F$ is Borel $\mc{K}$-structurable then so is $E$.
\end{proposition}

The following simple lemma, whose proof we omit, relates class-bijective Borel homomorphisms with Theorem \ref{thm:main}.

\begin{lemma} \label{lem:classbiject}
Let $G \cc X$ and $G \cc Y$ be Borel actions of $G$, let $E$ and $F$ be the induced orbit equivalence relations on $X$ and $Y$ respectively, and let $f: X \ra Y$ be a $G$-equivariant Borel map. Then $f$ is a homomorphism from $E$ to $F$, and if $G$ acts freely on both $X$ and $Y$ then $f$ is class-bijective.
\end{lemma}

Theorem \ref{thm:main}, Lemma \ref{lem:classbiject}, and Proposition \ref{prop:Kstruct} therefore imply that out of all equivalence relations coming from free actions of $G$, the equivalence relation $F(G,2)$, generated by $G\cc \mathrm{Free}(2^G)$, is the most difficult to structure in a Borel way.

\begin{corollary}\label{cor:structure}
Let $\mc{K}$ be a class of countable $L$-structures which is closed under isomorphism. Suppose that $F(G,2)$ is Borel $\mc{K}$-structurable. Then every equivalence relation generated by a free Borel action of $G$ is Borel $\mc{K}$-structurable.
\end{corollary}

This should be contrasted with Thomas's result \cite[Corollary 6.3]{Th12} that there are countable groups $G$, e.g., $G=SL_3(\Z )$, for which $F(G,2) <_B F(G,3)<_B \cdots <_B F(G,\N )$. Here $F(G,K )$ denotes the equivalence relation generated by $G\cc \mathrm{Free}(K^G )$, and $<_B$ denotes strict Borel reducibility. So, while Corollary \ref{cor:structure} shows that from the point of view of Borel structurability, $F(SL_3(\Z ),2)$ is the most complicated equivalence relation generated by a free action of $SL_3(\Z )$, Thomas's result shows that from the point of view of Borel reducibility this is not the case.

In \cite{Th09}, Thomas shows that Martin's conjecture implies that the Borel complexity of any weakly universal countable Borel equivalence relation must concentrate off of a conull set with respect to any Borel probability measure. In \cite{Ma13b}, Marks shows that the Borel complexity of any universal $\mc{K}$-structurable countable Borel equivalence relation is achieved on a null set with respect to any Borel probability measure. Along these lines, Theorem \ref{thm:main}.(2) implies that for any countable group $G$, the Borel-structurability complexity of $F(G,2)$ is achieved on a null set with respect to any Borel probability measure. In fact, rather than using the ideal of null sets of a Borel probability measure, we can obtain the same conclusion for a much wider class of ideals. For example, a sufficient condition on the ideal $I$ of $\mathrm{Free}(2^G)$ would be that every uncountable collection $C$ of pairwise-disjoint Borel subsets of $X$ satisfies $C\cap I \neq \varnothing$. The ideal of null sets for any Borel probability measure has this property, as does the ideal of meager sets for any compatible Polish topology on $\mathrm{Free}(2^G)$. Below we state yet a weaker requirement on the ideal.

In what follows, for a Polish space $Z$ we let $K(Z)$ denote the Polish space of all compact subsets of $Z$.

\begin{corollary}
Let $I$ be an ideal on $\mathrm{Free}(2^G)$. Assume that every nonempty perfect set $P\subseteq K(\mathrm{Free}(2^G))$ of pairwise disjoint $G$-invariant compact subsets of $\mathrm{Free}(2^G)$ satisfies $P \cap I \neq \varnothing$. Then there exists a compact $G$-invariant set $K\subseteq \mathrm{Free}(2^G)$ with $K\in I$ such that for any free Borel action $G\cc X$, there exists a $G$-equivariant class-bijective Borel map $f:X\ra K$.
\end{corollary}

\begin{proof}
By Theorem \ref{thm:main}.(2) there exists a family $\{ f_w \} _{w\in 2^\N}$ of $G$-equivariant class-bijective Borel maps $f_w:\mathrm{Free}(2^G) \ra 2^G$ with $\ol{f_w(\mathrm{Free}(2^G))}\subseteq \mathrm{Free}(2^G)$ and
$$\ol{f_w(\mathrm{Free}(2^G))} \cap \ol{f_z(\mathrm{Free}(2^G))} = \emptyset$$
for all distinct $w,z\in 2^\N$. Moreover, for each fixed $y \in \mathrm{Free}(2^G)$, the map $w\mapsto f_w(y)$ from $2^\N$ to $2^G$ is continuous. It follows that the map $2^\N \ra K(2^G)$ given by
$$w \mapsto \ol{f_w(\mathrm{Free}(2^G))}$$
is Borel. Therefore
$$\Big\{ \ol{f_w(\mathrm{Free}(2^G))} \Big\}_{w\in 2^\N }$$
is an uncountable analytic subset of $K(2^G )$, so there is a nonempty perfect subset $P\subseteq \big\{ \ol{f_w(\mathrm{Free}(2^G))} \big\} _{w\in 2^\N }$. Since $P\subseteq K(\mathrm{Free}(2^G))$ and since elements of $P$ are $G$-invariant and pairwise disjoint, we must have $P\cap I \neq \varnothing$. This shows that there is some $w_0\in 2^\N$ with
$$\ol{f_{w_0}(\mathrm{Free}(2^G))}\in I.$$
Let $K= \ol{f_{w_0}(\mathrm{Free}(2^G))}$. Then $K\in I$ and if $G\cc X$ is any free Borel action of $G$ then by Theorem \ref{thm:main} there exists a $G$-equivariant class-bijective Borel map $f:X\ra \mathrm{Free}(2^G)$, whence $f_{w_0}\circ f : X\ra K$ is a $G$-equivariant class-bijective Borel map to $K$.
\end{proof}

\subsection{Borel chromatic number}
By a {\bf graph} on a set $X$ we mean a symmetric irreflexive subset $\mc{G}$ of $X\times X$. Let $K$ be any set. Then a {\bf $K$-coloring} of $\mc{G}$ is a map $\kappa : X \ra K$ such that $\kappa (x)\neq \kappa (y)$ whenever $(x,y)\in \mc{G}$. Let $X$ be a standard Borel space and let $\mc{G}$ be a Borel graph on $X$, i.e., $\mc{G}$ is Borel as a subset of $X\times X$. The {\bf Borel chromatic number} of $\mc{G}$, denoted $\chi _B (\mc{G})$ is defined to be the minimum cardinality of a standard Borel space $K$ such that there exists a Borel $K$-coloring $\kappa :X\ra K$ of $\mc{G}$.

Let $G$ be a countable group and fix a subset $S$ of $G$. To each free Borel action $G\cc X$ of $G$ we associate the Borel graph
\[
\mc{G}_{X} = \{ (x,s\cdot x)\csuchthat x\in X, \ s\in S\cup S^{-1},\ s\neq 1_G  \} .
\]

\begin{corollary}\label{cor:color}
Let $G\cc X$ be a free Borel action of $G$ on a standard Borel space $X$. Then $\chi _B(\mc{G}_{X})\leq \chi _B (\mc{G}_{\mathrm{Free}(2^G)})$.
\end{corollary}

\begin{proof}
By Theorem \ref{thm:main} there exists a Borel $G$-equivariant map $f:X\ra \mathrm{Free}(2^G)$. Then any Borel $K$-coloring of $\mc{G}_{\mathrm{Free}(2^G)}$ pulls back, via $f$, to a Borel $K$-coloring of $\mc{G}_{X}$.
\end{proof}

This answers a question of Marks \cite[Question 3.10]{Ma13a}. By combining Corollary \ref{cor:color} with \cite[Theorem 1.2]{Ma13a} we conclude that for the free group $\F _n$ of rank $n$, with free generating set $S= \{ s_0,\dots , s_{n-1} \}$, we have $\chi _B (\mc{G}_{\mathrm{Free}(2^{\F _n})}) = 2n + 1$.

\subsection{Free factors and Shannon entropy}

Let $G\cc X$ be a Borel action of $G$. A {\bf generating partition for $G\cc X$} is a countable Borel partition $\mc{P}$ of $X$ such that the smallest $G$-invariant $\sigma$-algebra containing $\mc{P}$ is the entire Borel $\sigma$-algebra. Equivalently, $\mc{P}$ is generating if for every $x \neq y \in X$ there is $g \in G$ such that $\mc{P}$ separates $g \cdot x$ and $g \cdot y$. Let $\mu$ be a Borel probability measure on $X$. We say that $\mc{P}$ is a {\bf generating partition for $G\cc (X,\mu )$} if it is a generating partition for $G\cc X_0$ for some $G$-invariant conull $X_0\subseteq X$.  The {\bf Shannon entropy} of a countable partition $\mc{P}$ is given by
\[
H_\mu (\mc{P}) = -\sum _{P\in \mc{P}}\mu (P)\log (\mu (P)) .
\]

\begin{corollary} \label{cor:shannon}
Let $G\cc (X,\mu )$ be a free probability measure preserving action of $G$. Then for any $\epsilon >0$ there exists a factor map $f:(X,\mu ) \ra (Z,\eta )$ onto a free action $G\cc (Z,\eta )$ which admits a 2-piece generating partition $\{ C_0, C_1 \}$ with $H_\eta (\{ C_0, C_1 \} ) < \epsilon$.
\end{corollary}

In \cite{DP02}, Danilenko and Park proved this for amenable groups by using the Ornstein--Weiss quasi-tiling machinery \cite{OW}. They also obtained a similar result for torsion-free groups but with a countably infinite partition.

\begin{proof}
Since $H_\mu (\{ A,X\setminus A \} ) \ra 0$ as $\mu (A) \ra 0$, there exists an $r>0$ such that $\mu (A) <r \ \Ra \ H_\mu (\{ A,X\setminus A \} )<\epsilon$. Since the map $[A]_\mu \mapsto \mu(A)$ is a continuous function from $\mathrm{MALG}_\mu$ to $\R$, it follows from Theorem \ref{thm:main}.(3) that there is a Borel set $A \subseteq X$ with $\mu(A) < r$ such that the induced map $f_A : X \rightarrow 2^G$ has image $f_A(X) \subseteq \mathrm{Free}(2^G)$. Take $(Z, \eta) = (\mathrm{Free}(2^G), f_A(\mu))$, and let $\{C_0, C_1\}$ be the canonical generating partition of $2^G$, i.e. $C_i = \{w \in 2^G \csuchthat w(1_G) = i\}$ for $i \in \{0, 1\}$. Then $A = f_A^{-1}(C_1)$, whence $\eta(C_1) = \mu(A) < r$ and $\sH_\eta(\{C_0, C_1\}) < \epsilon$.
\end{proof}

\subsection{Rohklin's generator theorem}

In \cite{R67}, Rohklin proved that if $\Z \cc (X, \mu)$ is a probability measure preserving ergodic free action then its Kolmogorov--Sinai entropy, denoted $h_\Z(X, \mu)$, can be computed from the Shannon entropy of generating partitions by the formula
$$h_\Z(X, \mu) = \inf \Big\{ \sH_\mu(\alpha) \: \alpha \text{ is a countable generating partition for } \Z \cc (X, \mu) \Big\}.$$
Although much of the entropy theory of $\Z$-actions has been generalized to actions of countable amenable groups, such an extension of Rohklin's theorem has not appeared in the literature. This may be due to the fact that Rohklin's theorem is quite similar to, and appeared just prior to, the much more famous Krieger finite generator theorem \cite{Kr70}. Using Corollary \ref{cor:shannon}, we are able to provide a short proof of a generalized version of Rohklin's theorem (one could also obtain this generalization by using the methods in \cite{DP02}). While this result will not be surprising to experts on entropy theory, we believe that it is important to record it in the literature.

\begin{cor}[Rohklin's generator theorem]
Let $G$ be a countably infinite amenable group, and let $G \cc (X, \mu)$ be a probability measure preserving ergodic free action. Then the Kolmogorov--Sinai entropy of this action satisfies
$$h_G(X, \mu) = \inf \Big\{ \sH_\mu(\alpha) \: \alpha \text{ is a countable generating partition for } G \cc (X, \mu) \Big\}.$$
\end{cor}

\begin{proof}
A result of Jackson, Kechris, and Louveau \cite[Theorem 5.4]{JKL02} states that any aperiodic Borel action of a countable group has a countable generating partition. In particular $G\cc (X,\mu )$ has a countable generating partition. Furthermore, it is a well known property of Kolmogorov--Sinai entropy that $h_G(X, \mu) \leq \sH_\mu(\alpha)$ for every countable generating partition $\alpha$. So we immediately obtain an inequality, and when $h_G(X, \mu) = \infty$ we obtain the equality. So assume that $h_G(X, \mu) < \infty$ and fix $\epsilon > 0$. Apply Corollary \ref{cor:shannon} to obtain factor map $f: (X, \mu) \rightarrow (Z, \eta)$ onto a free action $G \cc (Z, \eta)$ which admits a generating partition $\mathcal{Q}'$ with $\sH_\eta(\mathcal{Q}') < \epsilon / 2$. In particular, we have the bound $h_G(Z, \eta) < \epsilon / 2$. By the Ornstein--Weiss theorem \cite{OW}, there is an essentially free action of $\Z$ on $(Z, \eta)$ such that the $\Z$-orbits and the $G$-orbits coincide on an invariant conull subset of $Z$, and moreover such that the entropy $h_\Z(Z, \eta)$ is $0$. The actions of $\Z$ and $G$ are related by a cocycle $\alpha : \Z \times Z \rightarrow G$ defined $\eta$-almost-everywhere by the rule
$$\alpha(k, z) = g \Longleftrightarrow k \cdot z = g \cdot z.$$
The action of $\Z$ lifts to an ergodic essentially free action on $(X, \mu)$. Specifically, the action of $\Z$ on $(X, \mu)$ is defined $\mu$-almost-everywhere by the rule
$$k \cdot x = g \cdot x \Longleftrightarrow \alpha(k, f(x)) = g.$$
Now the Rudolph--Weiss theorem \cite{RW} implies that
$$h_G(X, \mu) - h_G(Z, \eta) = h_\Z(X, \mu) - h_\Z(Z, \eta) = h_\Z(X, \mu).$$
Thus $h_\Z(X, \mu) \leq h_G(X, \mu)$.

Apply the original Rohklin generator theorem to obtain a generating partition $\mathcal{P}$ for $\Z \acts (X, \mu)$ with $\sH_\mu(\mathcal{P}) < h_\Z(X, \mu) + \epsilon / 2$. Pull back the partition $\mathcal{Q}'$ of $Z$ to get a partition $\mathcal{Q}$ of $X$. We claim that $\mathcal{P} \vee \mathcal{Q}$ is a generating partition for $G \acts (X, \mu)$. Verifying this claim will complete the proof since
$$\sH_\mu(\mathcal{P} \vee \mathcal{Q}) \leq \sH_\mu(\mathcal{P}) + \sH_\mu(\mathcal{Q}) < h_\Z(X, \mu) + \epsilon/ 2 + \epsilon / 2 \leq h_G(X, \mu) + \epsilon.$$
Let $X_0 \subseteq X$ be a $G$-invariant conull set such that: (i) the action of $\Z$ on $X_0$ is well-defined and related to the $G$-action via the cocycle $\alpha$; (ii) the partition $\mathcal{P}$ is a generating partition (in the purely Borel sense) for $\Z \acts X_0$; and (iii) the partition $\mathcal{Q}'$ is a generating partition for $G \acts f(X_0)$. Fix $x, y \in X_0$ with $x \neq y$. If there is $g \in G$ such that $g \cdot x$ and $g \cdot y$ are separated by $\mathcal{Q}$ then we are done. So we may suppose that $f(x) = f(y) \in Z$. Since $x \neq y \in X_0$ and $\mathcal{P}$ is a generating partition for $\Z \acts X_0$, there is $k \in \Z$ such that $\mathcal{P}$ separates $k \cdot x$ and $k \cdot y$. However, setting $g = \alpha(k, f(x)) = \alpha(k, f(y))$ we have that $k \cdot x = g \cdot x$ and $k \cdot y = g \cdot y$. Thus $g \cdot x$ and $g \cdot y$ are separated by $\mathcal{P}$. We conclude that $\mathcal{Q} \vee \mathcal{P}$ is generating for $G \acts X_0$.
\end{proof}

\section{Proof of Theorem 1.1}

\subsection{Preliminary Borel combinatorics}

\begin{lemma}[\cite{KST99}]\label{lem:ind}
Let $\mc{G}$ be a Borel graph on a standard Borel space $X$. Assume that every vertex of $\mc{G}$ has finite degree. Then there exists a maximal (with respect to inclusion) Borel $\mc{G}$-independent set.
\end{lemma}

The following Lemma will be used frequently.

\begin{lemma}\label{lem:S}
Let $G\cc X$ be a free Borel action of a countable group $G$ on the standard Borel space $X$. Let $S\subseteq G$ be finite and let $Y \subseteq X$ be Borel. Then there exists a maximal Borel set $D\subseteq Y$ having the property that $S\cdot y\cap S\cdot y' =\emptyset$ for all distinct $y,y' \in D$.
\end{lemma}

\begin{proof}
Apply Lemma \ref{lem:ind} to the Borel graph
\[
\mc{G} = \{ (y,y' ) \in Y\times Y \csuchthat y\neq y'\mbox{ and }S\cdot y\cap S\cdot y'\neq \emptyset \} .\qedhere
\]
\end{proof}

\begin{lemma}[\cite{KST99}]\label{lem:color}
Let $\mc{G}$ be a Borel graph on a standard Borel space $X$. Let $m\in \N$ and assume that every vertex of $\mc{G}$ has degree at most $m$. Then there exists a Borel $m+1$-coloring $\kappa :X \ra \{ 0,1,\dots , m \}$ of $\mc{G}$.
\end{lemma}

Recall that a subset $M \subseteq G$ is {\bf left (resp. right) syndetic} if there is a finite set $F \subseteq G$ with $F M = G$ (resp. $MF=G$). If $G \acts X$ is a free action, then call a subset $M \subseteq X$ {\bf locally syndetic} if for every $x\in X$ there exists a finite $F\subseteq G$ with $G\cdot x \subseteq F\cdot M$. Equivalently, for every $x\in X$ the set $\{ g \in G \csuchthat  g\cdot x\in M \}$ is left syndetic in $G$. Call $M\subseteq X$ {\bf (uniformly) syndetic} if there is a finite subset $F \subseteq G$ such that $F \cdot M = X$.

\begin{proposition}\label{prop:synd}
Let $G\cc X$ be free Borel action of $G$ a standard Borel space $X$.
\begin{enumerate}
\item If $P\subseteq X$ is a syndetic Borel subset of $X$ then there exists $M\subseteq P$ Borel such that $M$ and $P\setminus M$ are both syndetic.
\item There exists a sequence $\{ M_n\} _{n\in \N}$ of syndetic Borel subsets of $X$ which are pairwise disjoint.
\end{enumerate}
It follows that for any Borel probability measure $\mu$ on $X$ and any $\epsilon >0$ there exists a syndetic Borel subset $M\subseteq X$ with $\mu (M)<\epsilon$.
\end{proposition}

\begin{proof}
It suffices to show (1), since (2) then follows by induction. Fix $F\subseteq G$ finite with $F^{-1} \cdot P = X$. Then $F \cdot x \cap P \neq \varnothing$ for all $x \in X$. Let $Q$ be a finite symmetric subset of $G$ which properly contains $F$ and some disjoint translate $F g$ of $F$. Then $|Q \cdot x \cap P | \geq 2$ for all $x\in X$. Apply Lemma \ref{lem:S} to obtain a maximal Borel subset $M$ of $P$ with $Q\cdot x \cap Q\cdot y =\emptyset$ for all distinct $x,y\in M$. By maximality of $M$ we have $P \subseteq Q^2 \cdot M$. Thus $M$ is syndetic since $P$ is syndetic. In addition, $F g \cdot M$ is disjoint from $M$ and thus $(P \setminus M) \cap F g \cdot x \neq \emptyset$ for all $x \in M$. It follows that $M \subseteq g^{-1} F^{-1} \cdot (P \setminus M)$ and hence $P \setminus M$ is syndetic as well.
\end{proof}

\subsection{Notation}

In what follows it will be useful for us to deal with functions $X\ra \{ 0, 1\}$ instead of subsets of $X$ since we will often be working with partial functions $\phi :Y\ra \{ 0, 1\}$ defined only on some subset $Y\subseteq X$. Let $2^{\subseteq G}$ denote the set of all partial functions $w: \dom (w) \ra \{ 0, 1 \}$ with $\dom (w) \subseteq G$. Two partial functions are said to be {\bf compatible} if they agree on the intersection of their domains; they are called {\bf incompatible} otherwise. Given a partial function $\phi : \dom (\phi ) \ra \{ 0, 1 \}$ with $\dom (\phi ) \subseteq X$, we define $\wt{\phi} : X \ra 2^{\subseteq G}$ by
\[
\wt{\phi}(x)(g) =
\begin{cases}
\phi (g^{-1} \cdot x ) &\text{ if }g^{-1}\cdot x \in \dom (\phi ) , \\
\text{undefined}&\text{ if }g^{-1}\cdot x \not\in\dom (\phi ) .
\end{cases}
\]
When $\dom (\phi ) =X$ then $\wt{\phi } : X\ra 2^G$ is a $G$-equivariant map to the $2$-shift.

\begin{defn}\label{def:recognizable}
Let $G\cc X$ be an action of $G$ on a set $X$. Let $\phi: \dom (\phi ) \rightarrow \{0, 1\}$ be a partial function with $\dom (\phi )\subseteq X$. A set $R\subseteq X$ is called {\bf $\phi$-recognizable} if there exists a finite $T\subseteq G$ such that $\wt{\phi} (x)\res T$ and $\wt{\phi}(y) \res T$ are incompatible for all $x\in R$, $y\in X\setminus R$ .
\end{defn}

Note that if $R\subseteq X$ is $\phi$-recognizable then $R$ is $\phi '$-recognizable for every $\phi '$ which extends $\phi$. We record the following useful lemma whose proof is straight-forward.

\begin{lemma} \label{lem:algebra}
Let $G \cc X$ be an action of $G$ on a set $X$, and let $\phi : \dom(\phi) \rightarrow \{0, 1\}$ be a partial function with $\dom(\phi) \subseteq X$. Then the collection of sets $R \subseteq X$ which are $\phi$-recognizable is a $G$-invariant algebra of subsets of $X$.
\end{lemma}

If $\dom (\phi ) = X$ then a set $R\subseteq X$ is $\phi$-recognizable if and only if $R= \wt{\phi }^{-1} (C)$ for some clopen $C\subseteq 2^G$. More generally, we have

\begin{proposition}\label{prop:rec}
A set $R\subseteq X$ is $\phi$-recognizable if and only if there exists a clopen $C\subseteq 2^G$ such that
\begin{equation}\label{eqn:extends}
R = \{ x\in X\csuchthat (\exists f \in C )(f\text{ extends }\wt{\phi}(x)) \} .
\end{equation}
\end{proposition}

\begin{proof}
If $R$ is $\phi$-recognizable as witnessed by the finite set $T\subseteq G$, then the set $C = \{ f\in 2^G \csuchthat (\exists x\in R )(f\text{ extends }\wt{\phi}(x)\res T)\}$ is clopen and (\ref{eqn:extends}) is immediate. Conversely, if $C\subseteq 2^G$ is a clopen set satisfying (\ref{eqn:extends}), then any finite set $T\subseteq G$ for which $C$ is $w\mapsto w\res T$-measurable witnesses that $R$ is $\phi$-recognizable.
\end{proof}

\subsection{Outline of the construction}\label{sec:outline}

The construction we use to prove Theorem \ref{thm:main} is based on methods from \cite[Chapter 10]{GJS12}. In \cite{GJS12}, Gao, Jackson, and Seward studied methods for constructing points $x \in 2^G$ such that the closure of the orbit of $x$ is contained in $\mathrm{Free}(2^G)$. This property is in fact equivalent to not only requiring that $x$ have trivial stabilizer but that all translates $g \cdot x$ of $x$ have trivial stabilizer in a certain local and uniform sense. Their methods therefore seem well suited for using local Borel algorithms for constructing equivariant Borel maps into $\mathrm{Free}(2^G)$. Using the methods from \cite{GJS12} comes at a price -- the construction is long and technical; but it also has its rewards -- in addition to obtaining $G$-equivariant Borel maps into $\mathrm{Free}(2^G)$, we also obtain items (1), (2), and (3) of Theorem \ref{thm:main}. We do not know if there is a shorter proof for simply obtaining a $G$-equivariant Borel map into $\mathrm{Free}(2^G)$.

We will sketch the proof of part (1) of Theorem \ref{thm:main} as it is a bit simpler than part (2). The proof of Theorem \ref{thm:main}.(1) is built off of an inductive argument. The inductive step is based on the following fact. Fix a non-identity group element $s \in G$, and suppose that $\phi : (X \setminus M) \rightarrow \{0, 1\}$ is a Borel function with $M \subseteq X$ a Borel syndetic set. Then there is a Borel syndetic set $M' \subseteq M$ and a Borel extension $\phi' : (X \setminus M') \rightarrow \{0, 1\}$ of $\phi$ having the property that for every $x \in X$, there is $g \in G$ with $g \cdot x, g s \cdot x \not\in M'$ and $\phi'(g \cdot x) \neq \phi'(g s \cdot x)$. This last property implies that for any equivariant map $f : X \rightarrow 2^G$ extending $\wt{\phi'}$, we will have $f(x) \neq f(s \cdot x) = s \cdot f(x)$ for all $x \in X$. Thus $s \not\in \Stab(f(x))$ for every $x \in X$. Theorem \ref{thm:main}.(1) is then proved by repeatedly applying the above fact for each non-identity $s \in G$.

It remains to sketch a proof of the above fact. By using the syndeticity of $M$, we simultaneously define an extension $\phi ^*$ of $\phi$ while building a syndetic Borel set $\Delta \subseteq X$ which is $\phi^*$-recognizable. Creating a recognizable $\Delta$ takes a substantial amount of work, but roughly speaking this task is achieved by assigning a value of $1$ to many points in $M$ near $\Delta$ so that points in $\Delta$ locally see a high density of $1$'s nearby while points in $X \setminus \Delta$ locally see a lower density of $1$'s nearby. We furthermore build $\Delta$ so that each $\delta \in \Delta$ has its own proprietary region $F \cdot \delta$, so that $F \cdot \delta \cap F \cdot \delta' = \varnothing$ for $\delta \neq \delta' \in \Delta$. Additionally, each region $F \cdot \delta$ will contain many points in $M \setminus \dom(\phi^*)$. We then extend $\phi^*$ to $\phi'$ by labeling the previously unlabelled points in $M \cap F \cdot \Delta$ so that distinct points $\delta \neq \delta' \in \Delta$ which are ``close'' to one another have distinct labellings of their $F$-regions.

Next we check that $\phi'$ has the desired property with respect to $s$. Let $W \subseteq G$ be finite with $W^{-1} \cdot \Delta = X$. Fix $x \in X$. Let $g \in W$ be such that $g \cdot x \in \Delta$. If $g s \cdot x \not\in \Delta$ then we are done since $\Delta$ is $\phi'$ recognizable. So suppose that $g s \cdot x \in \Delta$. Then setting $\delta = g \cdot x$ and $\delta' = g s \cdot x$ we have that
$$\delta' = g s \cdot x = (g s g^{-1}) \cdot (g \cdot x) = g s g^{-1} \cdot \delta \in W s W^{-1} \cdot \delta.$$
So by using the condition $\delta' \in W s W^{-1} \cdot \delta$ as our definition of ``close'' we have that there is $f \in F$ with $\phi'(f g \cdot x) \neq \phi'(f g s \cdot x)$. This completes the sketch.

We mention that a key point we will use in our proof is that the number of $\delta' \in \Delta$ which are ``close'' to a fixed $\delta \in \Delta$ will be bounded above by a quadratic polynomial of $|F|$, while the number of points in $F \cdot \delta \cap (M \setminus \dom(\phi^*))$ will be bounded below by a linear function of $|F|$. Thus for $|F|$ sufficiently large we have
$$2^{|F \cdot \delta \cap (M \setminus \dom(\phi^*))|} > | \{ \delta' \in \Delta \: \delta' \text{ is ``close'' to } \delta\} |.$$
The above inequality is what allows us to construct $\phi'$ as described. We point out that the freeness of $G \cc X$ is critical to this argument. If the action were non-free then $|F^2 \cdot x|$ could grow exponentially in terms of $|F \cdot x|$. We therefore do not know if there is a $G$-equivariant class-bijective Borel map $f : X \rightarrow 2^G$ for general aperiodic Borel actions $G \cc X$.

\subsection{The construction}\label{sec:construction}

\begin{lem} \label{COUNTING LEM}
Let $G$ be a countably infinite group. Let $B, C \subseteq G$ be finite, and let $r > 0$. Then there exist finite sets $\Lambda \subseteq F \subseteq G$ such that
\begin{enumerate}
\item[\rm (i)] $C \subseteq F$;
\item[\rm (ii)] $B \cdot \Lambda \subseteq F$;
\item[\rm (iii)] $B \cdot \lambda \cap B \cdot \lambda' = \varnothing$ for all $\lambda \neq \lambda' \in \Lambda$;
\item[\rm (iv)] $B \cdot \Lambda \cap C = \varnothing$;
\item[\rm (v)] $|\Lambda| \geq \log_2(r \cdot |F|^2) + r$.
\end{enumerate}
\end{lem}

\begin{proof}
Pick $n \in \N$ satisfying
$$n \geq \log_2 \Big( r \cdot (|C| + n \cdot |B|)^2 \Big) + r.$$
Such an $n$ exists since the right-hand side is a sub-linear function of $n$. Now since $G$ is infinite and $B$ and $C$ are finite, we can find $n$ group elements $\lambda_1, \lambda_2, \ldots, \lambda_n \in G$ such that $B \cdot \lambda_i \cap B \cdot \lambda_j = \varnothing$ for all $i \neq j$ and $B \cdot \lambda_i \cap C = \varnothing$ for all $i$. Set
$$\Lambda = \{\lambda_1, \lambda_2, \ldots, \lambda_n\} \quad \text{ and } \quad F = C \cup B \cdot \Lambda.$$
Then properties (i) through (iv) are immediate, and (v) follows from our choice of $n$.
\end{proof}

\begin{lemma} \label{IND STEP}
Let $G$ be a countably infinite group and let $G \acts X$ be a free Borel action. Let $M, R \subseteq X$ be Borel sets and let $\phi : X \setminus (M \cup R) \rightarrow \{0, 1\}$ be a Borel function. Assume that $M$ and $R$ are disjoint, $M$ is syndetic, and $R$ is $\phi$-recognizable. Fix any $s \in G$ with $s\neq 1_G$. Then there are Borel sets $M', R' \subseteq X$ and a Borel function $\phi' : X \setminus (M' \cup R' \cup R) \rightarrow \{0, 1\}$ such that
\begin{enumerate}
\item[\rm (i)] $M'$ and $R'$ are disjoint subsets of $M$;
\item[\rm (ii)] $\phi'$ extends $\phi$;
\item[\rm (iii)] $M'$ and $R'$ are both syndetic and $R'$ is $\phi'$-recognizable;
\item[\rm (iv)] There is a finite $T\subseteq G$ such that for all $x\in X$ the partial functions $\wt{\phi '}(x)\res T$ and $\wt{\phi '}(s\cdot x)\res T$ are incompatible.
\end{enumerate}
\end{lemma}

\begin{proof}
The most difficult part of this proof is to extend $\phi$ in order to build recognizable syndetic subsets of $X$. As we know nothing of $\phi$ aside from its domain and the recognizability of $R$, which may be empty, $\phi$ is essentially a noisy background to which we must somehow add some recognizability. This involves several steps of coding techniques. The first step involves a crude process of counting the number of $1$'s which appear in certain regions. Specifically, since $M$ is syndetic, there is a finite set $A \subseteq G$ so that for every $x \in X$ we have $|A \cdot x \cap M| \geq 2$. For any Borel $Y \subseteq X$ and any Borel function $\theta : Y \rightarrow \{0, 1\}$ define the counting function $c_\theta$ by
\begin{equation}\label{eqn:c}
c_\theta (x) = |\{a \in A \: a \cdot x \in \dom(\theta )\setminus R \text{ and } \theta (a \cdot x) = 1 \} |.
\end{equation}
Note that if $R$ is $\theta$-recognizable and if $X_0\subseteq X$ is any $\theta$-recognizable set with $A\cdot X_0 \subseteq \dom (\theta ) \cup R$, then the set $\{x \in X_0 \csuchthat a \cdot x \not\in R \text{ and } \theta(a \cdot x) = i\}$ is $\theta$-recognizable for $i = 0, 1$ and $a \in A$. Therefore by Lemma \ref{lem:algebra} the set $\{ x\in X_0 \csuchthat c_\theta (x) = i \}$ is $\theta$-recognizable for all $i\in \N$.

Set $N = |A|-2$ and note that $c_\phi(x) \leq N$ for all $x \in X$. We will soon carefully add in $1$'s at select locations with the intention of creating local maximums for the counting function $c$. If we add in some $1$'s in $A \cdot x$, then these new $1$'s will be visible from $A^{-1} A \cdot x$. We therefore use $B = A^{-1} A$ as a buffer region and we will frequently require that points $x, y \in X$ have disjoint $B$-regions, meaning $B \cdot x \cap B \cdot y = \varnothing$. A fact which we will use repeatedly is that $B = B^{-1}$.

We will soon add in values of $1$ at select locations in order to create local maximums for the counting function $c$, but we must first decide how far apart we want these local maximums to be. We will need a verification set $V \subseteq G$ and a verification function $v: B \times B \rightarrow G$ whose significance will become clear later. Let $v : B \times B \rightarrow G$ be any function satisfying the following for all $(b_1, b_2), (b_3, b_4) \in B \times B$:
$$v(b_1, b_2) = v(b_2, b_1);$$
$$B \cdot v(b_1, b_2) \cdot b_1 \cap B = \varnothing;$$
$$(b_1, b_2) \neq (b_3, b_4) \Longrightarrow B \cdot v(b_1, b_2) \cdot b_1 \cap B \cdot v(b_3, b_4) \cdot b_3 = \varnothing.$$
Such a function $v$ exists since $B$ is finite and $G$ is infinite. Set
$$V = \bigcup_{(b_1, b_2) \in B \times B} B \cdot v(b_1, b_2).$$
Now pick finite sets $\Lambda \subseteq F \subseteq G$ such that
\begin{enumerate}
\item[\rm (a)] $B^3 \cup V B \subseteq F$;
\item[\rm (b)] $B \Lambda \subseteq F$;
\item[\rm (c)] $B \cdot \lambda \cap B \cdot \lambda' = \varnothing$ for all $\lambda \neq \lambda' \in \Lambda$;
\item[\rm (d)] $B \Lambda B \cap (B \cup V B) = \varnothing$;
\item[\rm (e)] $|\Lambda| \geq \log_2(2 |B|^{3N+3} |F|^2 + 1) + 2 \log_2(|B|) + 4$.
\end{enumerate}
Such sets $\Lambda, F \subseteq G$ exist by Lemma \ref{COUNTING LEM}.

Now we decide on the locations where we will create local maximums for the counting function $c$. In choosing such locations, we wish to favor locations $x$ where $c_\phi(x)$ is already large. Apply Lemma \ref{lem:S} to obtain a maximal Borel subset $D_0$ of $\{x \in X \: c_\phi(x) = N\}$ having the property that $F B \cdot d \cap F B \cdot d' = \varnothing$ for all $d \neq d' \in D_0$. Next, let $D_1$ be a maximal Borel subset of
$$\{x \in X \: c_\phi(x) = N - 1\} \setminus B^4 F^{-1} F B \cdot D_0$$
having the property that $F B \cdot d \cap F B \cdot d' = \varnothing$ for all $d \neq d' \in D_1$. In general, once $D_0$ through $D_{m-1}$ have been defined with $m \leq N$, let $D_m$ be a maximal Borel subset of
$$\{x \in X \: c_\phi(x) = N - m\} \setminus \bigcup_{i=0}^{m-1} B^{3m+1} F^{-1} F B \cdot D_i$$
having the property that $F B \cdot d \cap F B \cdot d' = \varnothing$ for all $d \neq d' \in D_m$. This defines $D_0, D_1, \ldots, D_N$. Set $D = \bigcup_{0 \leq i \leq N} D_i$. We point out a few important properties of $D$.

(1). Let $x \in X$ and suppose that $c_\phi(x) = N - m$. Then by the maximal property of $D_m$ either $F B \cdot x \cap F B \cdot D_m \neq \varnothing$ or else
$$x \in \bigcup_{i=0}^{m-1} B^{3m+1} F^{-1} F B \cdot D_i.$$
In either case we have
$$x \in \bigcup_{i=0}^m B^{3m+1} F^{-1} F B \cdot D_i.$$

(2). For every $x \in X$ there is $0 \leq m \leq N$ with $c_\phi(x) = N - m$. Therefore from (1) it follows that
$$X \subseteq B^{3N+1} F^{-1} F B \cdot D.$$
In particular, $D$ is syndetic.

(3). For $0 \leq m \leq N$, $d \in D_m$, and $x \in B^3 \cdot d$, we have $c_\phi(x) \leq c_\phi(d) = N - m$. This says that each point $d \in D_m$ achieves a local maximum for the function $c_\phi$ in the region $B^3 \cdot d$. We prove this claim by contradiction. Towards a contradiction, suppose that $t < m$ and $c_\phi(x) = N - t$. Then by (1)
$$x \in \bigcup_{i=0}^t B^{3t+1} F^{-1} F B \cdot D_i \subseteq \bigcup_{i=0}^{m-1} B^{3m-2} F^{-1} F B \cdot D_i.$$
Therefore
$$d \in B^3 \cdot x \subseteq \bigcup_{i=0}^{m-1} B^{3m+1} F^{-1} F B \cdot D_i,$$
which contradicts the definition of $D_m$ and the fact that $d \in D_m$.

We will now extend $\phi$ to $\phi_1$. The purpose of $\phi_1$ is to place extra $1$'s near the select locations $D \subseteq X$. We define $\phi_1$ to be an extension of $\phi$ with
$$\dom(\phi_1) = \dom(\phi) \cup (M \cap B \cdot D)$$
and with the property that for every $d \in D$ all elements of $M \cap B \cdot d$ are assigned the value $0$ except for precisely $2$ elements in $M \cap A \cdot d$ which are assigned the value $1$. Such a function $\phi_1$ exists since $B \cdot d \cap B \cdot d' = \varnothing$ for all $d \neq d' \in D$, $A \subseteq B$, and $|M \cap A \cdot d| \geq 2$ for all $d \in D$. Observe that for $x \in X$
\begin{align*}
& c_\phi(x) \leq c_{\phi_1}(x) \leq c_\phi(x) + 2,\\
\text{and} \quad & c_{\phi_1}(x) > c_\phi(x) \Longrightarrow x \in B \cdot D.
\end{align*}

The function $\phi_1$ has the nice property that for $d \in D_m$ we have $c_{\phi_1}(d) = N - m + 2$, for $x\in B\cdot d$ we have $c_{\phi_1}(x) \leq N - m + 2$, and for $y\in B^3\cdot d\setminus B\cdot d$ we have $c_{\phi_1}(y) \leq N - m$. We want $D$, or at least a set close to $D$, to become recognizable for some extension of $\phi_1$. Creating local maximums for the counting function $c$ was a crude first attempt, but a problem with $\phi_1$ is that there may be $d \in D_m$ and $d \neq x \in B \cdot d$ with $c_{\phi_1}(x) = c_{\phi_1}(d) = N - m + 2$. So in terms of locally maximizing $c_{\phi_1}$, $x$ and $d$ are in a tie. So we now introduce a tie-breaker by using the verification function $v$ and the verification set $V$. We extend $\phi_1$ to $\phi_2$ where $\phi_2$ has domain
$$\dom(\phi_2) = \dom(\phi_1) \cup \Big(M \bigcap A \cdot \{v(b_1, b_2) \cdot b_1 \: b_1 \neq b_2 \in B\} \cdot D \Big).$$
We require for each $d \in D$ and each $b_1 \neq b_2 \in B$ that $\phi_2$ have distinct behavior on the two regions $A \cdot v(b_1, b_2) b_1 \cdot d$ and $A \cdot v(b_1, b_2) \cdot b_2 \cdot d$. Specifically, for each $d \in D$ and each $b_1 \neq b_2 \in B$ we require that there be $a \in A$ such that either
$$\chi_R \big(a \cdot v(b_1, b_2) \cdot b_1 \cdot d \big) \neq \chi_R \big(a \cdot v(b_1, b_2) \cdot b_2 \cdot d \big),$$
where $\chi_R$ is the characteristic function of $R$, or else both $a \cdot v(b_1, b_2) \cdot b_1 \cdot d$ and $a \cdot v(b_1, b_2) \cdot b_2 \cdot d$ are in the domain of $\phi_2$ and
$$\phi_2 \big(a \cdot v(b_1, b_2) \cdot b_1 \cdot d \big) \neq \phi_2 \big(a \cdot v(b_1, b_2) \cdot b_2 \cdot d \big).$$
We further require that this be achieved while creating very few new $1$'s, meaning that for each $d \in D$ and $b_1 \neq b_2 \in B$
\begin{align*}
 & \ c_{\phi_2} \big(v(b_1, b_2) b_1 \cdot d \big) + c_{\phi_2} \big(v(b_1, b_2) b_2 \cdot d \big)\\
\leq & \ c_{\phi_1} \big(v(b_1, b_2) b_1 \cdot d \big) + c_{\phi_1} \big(v(b_1, b_2) b_2 \cdot d \big) + 1.
\end{align*}
We point out that for $b_1, b_2 \in B$ we have $A \cdot v(b_1, b_2) \cdot b_1 \subseteq V B \subseteq F$, and since $F \cdot d \cap F \cdot d' = \varnothing$ for each $d \neq d' \in D$, achieving these conditions is an independent local requirement for each $d \in D$. So if there is any such function $\phi_2$ then it can certainly be chosen to be Borel. By the definition of $v$, for every $d \in D$ and $b_1, b_2 \in B$ we have that $A \cdot v(b_1, b_2) \cdot b_1 \cdot d \cap B \cdot D = \varnothing$, so $\phi_1$ and $\phi$ are identical on $A \cdot v(b_1, b_2) \cdot b_1 \cdot d$ and hence
$$|A \cdot v(b_1, b_2) \cdot b_1 \cdot d \cap M \cap (X \setminus \dom(\phi_1))| = |A \cdot v(b_1, b_2) \cdot b_1 \cdot d \cap M| \geq 2.$$
Furthermore, for every $d \in D$ and $b_1, b_2, b_3, b_4 \in B$ the definition of $v$ implies that
$$\{b_1, b_2\} \neq \{b_3, b_4\} \Longrightarrow B \cdot v(b_1, b_2) \cdot b_1 \cdot d \cap B \cdot v(b_3, b_4) \cdot b_3 \cdot d = \varnothing.$$
Therefore one can achieve the conditions for $\phi_2$ by considering each $d \in D$ and each un-ordered pair $\{b_1, b_2\} \subseteq B$, $b_1 \neq b_2$, one at a time. For $d \in D$ and $b_1 \neq b_2 \in B$ we can find $a \in A$ with
$$a \cdot v(b_1, b_2) \cdot b_1 \cdot d \in M \cap (X \setminus \dom(\phi_1)).$$
Then $a \cdot v(b_1, b_2) \cdot b_1 \cdot d \not\in R$ since $M$ and $R$ are disjoint. If $a \cdot v(b_1, b_2) \cdot b_2 \cdot d \in R$ then we are done, and otherwise we can assume that $\phi_2(a \cdot v(b_1, b_2) \cdot b_2 \cdot d)$ is defined and set
$$\phi_2 \big( a \cdot v(b_1, b_2) \cdot b_1 \cdot d \big) = 1 - \phi_2 \big( a \cdot v(b_1, b_2) \cdot b_2 \cdot d \big).$$
We conclude that such a function $\phi_2$ exists, and that it can be chosen to be Borel. We note that $\phi_2$ satisfies the following for every $x \in X$:
$$c_\phi(x) \leq c_{\phi_2}(x) \leq c_\phi(x) + 2;$$
$$c_{\phi_2}(x) > c_\phi(x) \Longrightarrow x \in \left(B \cup \bigcup_{b_1 \neq b_2 \in B} B \cdot v(b_1, b_2) \cdot b_1 \right) \cdot D;$$
$$\text{and} \quad c_{\phi_2}(x) > c_\phi(x) + 1 \Longrightarrow x \in B \cdot D.$$

Now to complete the role of the verification set $V$, we extend $\phi_2$ to $\phi_3$ where
$$\dom(\phi_3) = \dom(\phi_2) \cup (M \cap V B \cdot D)$$
and for all $x \in \dom(\phi_3) \setminus \dom(\phi_2)$ we set $\phi_3(x) = 0$. Since we only added in new $0$'s, $\phi_3$ has all of the properties of $\phi_2$ listed above. We can now describe the tie-breaking procedure referred to earlier. For $Y \subseteq X$ and a function $\theta : Y \rightarrow \{0, 1\}$ which recognizes $R$, we associate to each $x \in X$ with $V \cdot x \subseteq Y \cup R$ the function $L_\theta (x) \in 3^V$ given by
$$L_\theta(x)(w) = \begin{cases}
2 & \text{if } w \cdot x \in R\\
\theta(w \cdot x) & \text{otherwise},
\end{cases}$$
i.e., $\dom (L_\theta ) = \{ x\in X\csuchthat V\cdot x \subseteq Y\cup R \}$, and for $w\in V$ we have $L_\theta (x)(w) = 2$ whenever $w \cdot x \in R$ and $L_\theta(x)(w) = \wt{\theta}(x)(w^{-1})$ otherwise. Note that if $X_0\subseteq X$ is any $\theta$-recognizable set with $V\cdot X_0 \subseteq Y\cup R$, then, since $\theta$ recognizes $R$, for each $w \in V$ and $i \in \{0, 1, 2\}$ the set
$$\{x \in X_0 \: L_\theta(x)(w) = i\}$$
is $\theta$-recognizable. We will work with extensions $\theta$ of $\phi_3$ so that $R$ will be $\theta$-recognizable automatically. The definition of $\phi_2$ guarantees that if $d \in D$ and $b_1 \cdot d \neq b_2 \cdot d \in B \cdot d$ then $L_\theta(b_1 \cdot d) \neq L_\theta(b_2 \cdot d)$, specifically
$$\exists a \in A \quad L_\theta(b_1 \cdot d) \big( a \cdot v(b_1, b_2) \big) \neq L_\theta(b_2 \cdot d) \big( a \cdot v(b_1, b_2) \big).$$
So if we fix a total ordering, denoted $\preceq$, of $3^V$ then we can pair each $d \in D$ with a unique element $p(d) = \delta \in B \cdot d$ as follows. For $d \in D_m$ we define $p(d) = b \cdot d$ where $b$ is the unique element of
$$S = \{b' \in B \: c_{\phi_3}(b' \cdot d) = N - m + 2\}$$
with $L_{\phi_3}(b \cdot d) \succeq L_{\phi_3}(b' \cdot d)$ for all $b' \in S$. The definition of $\phi_2$ guarantees that there is a unique $b$ satisfying this condition. We define $\Delta = p(D)$ and $\Delta_m = p(D_m)$ for $0 \leq m \leq N$. Note that since $F B \cdot d \cap F B \cdot d' = \varnothing$ for all $d \neq d' \in D$ it follows that $F \cdot \delta \cap F \cdot \delta' = \varnothing$ for all $\delta \neq \delta' \in \Delta$.

The Borel set $\Delta \subseteq X$ will play an important role in the remainder of this proof. This set is not necessarily $\phi_3$-recognizable, but we will soon make it recognizable for an extension of $\phi_3$. Before doing so we first drastically reduce the number of points in $X \setminus R$ which do not have an assigned value. Recall from earlier the set $\Lambda \subseteq F$, which satisfies properties (a) through (e). Enumerate $\Lambda$ as $\lambda_1, \lambda_2, \ldots, \lambda_\ell$. Let $K$ be the least integer greater than $\log_2(|B|)$. Note that $\ell - 2 K - 2 \geq \log_2(2 |B|^{3N+3} |F|^2 + 1)$.
Since $\phi_3$ and $\phi$ agree on $X \setminus (V B \cup B) \cdot D$, property (d) implies that for each $\delta \in \Delta$ and $1 \leq i \leq \ell$
$$|(X \setminus \dom(\phi_3)) \cap M \cap A \cdot \lambda_i \cdot \delta| = |M \cap A \cdot \lambda_i \cdot \delta| \geq 2.$$
We let $\phi_4$ be the Borel function which extends $\phi_3$ and satisfies:
$$X \setminus (R \cup A \cdot \Lambda \cdot \Delta) \subseteq \dom(\phi_4);$$
$$\forall \delta \in \Delta \ \forall 1 \leq i \leq \ell \quad |M \cap (X \setminus \dom(\phi_4)) \cap A \cdot \lambda_i \cdot \delta| = 1;$$
$$\text{and} \quad \forall x \in \dom(\phi_4) \setminus \dom(\phi_3) \quad \phi_4(x) = 0.$$
It follows from this definition, the properties of $\phi_3$, and properties (b) and (c) that for every $x\in X$:
\begin{equation}\label{eqn:line1}
|\{ a\in A \csuchthat a\cdot x \in M\setminus \dom (\phi _4) \} | \leq 1;
\end{equation}
\begin{equation}\label{eqn:line2}
\{ a\in A \csuchthat a\cdot x \in M\setminus \dom (\phi _4) \} \neq \emptyset \Longrightarrow c_{\phi _4}(x) = c_{\phi}(x) \text{ and }x\in B\Lambda \cdot \Delta ;
\end{equation}
$$c_\phi(x) \leq c_{\phi _4}(x) \leq c_\phi(x) + 2;$$
$$c_{\phi _4} (x) > c_\phi(x) \Longrightarrow x \in (B \cup V B) \cdot D;$$
$$c_{\phi _4}(x) > c_\phi(x) + 1 \Longrightarrow x \in B \cdot D$$
We have previously used two coding techniques -- creating local maximums in the counting function $c$, and using the verification set $V$ as a tie-breaker. We now employ a third technique which involves, for each $d \in D$ and $\delta = p(d) \in \Delta$, coding the element $b \in B$ satisfying $\delta = b \cdot d$. This is the final step in making $\Delta$ recognizable. It is true that $\Delta_0$ is $\phi_4$-recognizable since any $x \in X$ satisfying $c_{\phi_4}(x) = N + 2$ must lie in $B \cdot D_0$. However, for $0 < m \leq N$ the set $\Delta_m$ may not yet be $\phi_4$-recognizable since there are likely many points $x$ not lying in $B \cdot D_m$ which satisfy $c_{\phi_4}(x) = N - m + 2$. The key fact which we must use is that $D_m$ is carefully spaced from $D_t$ for $t < m$, and to use this information we must be able to backtrack from each $\delta \in \Delta$ to the $d \in D$ with $p(d) = \delta$. This is where our next coding technique comes in. For $i, k \in \N$ let $\bB_i(k) \in \{0, 1\}$ be the $i^\text{th}$ digit in the binary representation of $k$ (where $\bB_i(k) = 0$ when $2^{i-1}>k$). Fix an injective function $r : B \rightarrow \{0, 1, \ldots, 2^K - 1\}$. We extend $\phi_4$ to $\phi_5$ so that
$$\dom(\phi_5) = \dom(\phi_4) \cup (M \cap A \{\lambda_1, \lambda_2, \ldots, \lambda_K\} \cdot \Delta)$$
and for every $\delta \in \Delta$ and $1 \leq i \leq K$
$$c_{\phi_5}(\lambda_i \cdot \delta) \equiv \bB_i(r(b)) \mod 2$$
where $b^{-1} \cdot \delta = d \in D$ (or equivalently $\delta = p(d) = b \cdot d$).

We now formally check that the coding of the previous paragraph works, in the sense that, for every $0 \leq m \leq N$, $\Delta_m$ is $\phi_5$-recognizable if and only if $D_m$ is $\phi_5$-recognizable. Fix $0 \leq m \leq N$, and first suppose that $\Delta_m$ is $\phi _5$-recognizable. Then for each $1\leq i\leq K$ and $j\in \N$ the set $\{ y\in \Delta _m \csuchthat c_{\phi _5} (\lambda _i \cdot y ) = j \}$ is $\phi _5$-recognizable (since $(A\{ \lambda _1,\lambda _2,\dots ,\lambda _K\}\cdot \Delta_m )\subseteq \dom (\phi _5)\cup R$; see the remark immediately following (\ref{eqn:c})). To show that $D_m$ is $\phi _5$-recognizable it therefore suffices to show that for $x \in X$, $x \in D_m$ if and only if
\begin{quote}
there is some $b \in B$ with $b \cdot x \in \Delta_m$ such that for every $1 \leq i \leq K$, $$c_{\phi _5} (\lambda_i b \cdot x) \equiv \bB_i(r(b)) \mod 2.$$
\end{quote}
Clearly the above condition holds whenever $x \in D_m$. So suppose that $x \in X$ satisfies the stated property. Let $b \in B$ be such that $b \cdot x = \delta \in \Delta_m$, let $d \in D_m$ be such that $p(d) = \delta$, and let $b' \in B$ be such that $\delta = b' \cdot d$. Then the definition of $\phi_5$ together with the assumptions on $x$ imply that $\bB_i(r(b)) \equiv c_{\phi _5}(\lambda_i \cdot \delta) \equiv \bB_i(r(b')) \mod 2$ for all $1\leq i \leq K$. Since $r$ is injective and $K > \log_2(|B|)$ we obtain $b = b'$ and hence $x = b^{-1} \cdot \delta = d \in D_m$.

Now suppose that $D_m$ is $\phi _5$-recognizable. Then $\{ x\in B\cdot D_m \csuchthat c_{\phi _5}(x)= N-m+2 \}$ is $\phi _5$-recognizable, and moreover this set is contained in $\dom (L_{\phi _5})$. Therefore, $\phi _5$-recognizability of $\Delta _m$ will follow once we show that for $x \in X$, $x \in \Delta_m$ if and only if
\begin{quote}
$c_{\phi _5}(x) = N - m + 2$ and there is some $b\in B$ with $b\cdot x \in D_m$ such that for all $y \in B b \cdot x$, if $c_{\phi _5} (y) = N - m + 2$ then $L_{\phi _5}(y)\preceq L_{\phi _5}(x)$.
\end{quote}
The definition of $\Delta_m$ implies that the above conditions hold whenever $x \in \Delta_m$. So suppose that $x \in X$ satisfies the above condition, and let $b \in B$ be as described in the condition. Then $b \cdot x = d \in D_m$. Set $\delta = p(d) \in \Delta_m$. By the definition of the function $p$, we have that $L_{\phi _5}(x)\preceq L_{\phi _5}(\delta)$ and $c_{\phi _5}(\delta) = N - m + 2$. However, $\delta \in B \cdot d = B b \cdot x$, so the assumption on $x$ implies that $L_{\phi _5}(\delta ) \preceq L_{\phi _5}(x)$. The construction of $\phi_2$ guarantees that $L_{\phi _5}(z) \neq L_{\phi _5}(z')$ for all $z \neq z' \in B \cdot d$, and since $\preceq$ is a total ordering and $x, \delta \in B \cdot d$ we conclude that $x = \delta \in \Delta_m$.

In a moment we will verify that $\Delta$ is $\phi_5$-recognizable, but first we prove the following important claim:
\begin{quote}
$(\star )$ There is a finite set $T \subseteq G$ so that for all $x, y \in X$, if $c_{\phi_5}(x) = c_\phi(x) + 2$, $c_{\phi_5}(y) \leq c_\phi(y) + 1$, and $c_\phi(y) \leq c_\phi(x)$ then $\wt{\phi_5}(x) \res T$ and $\wt{\phi_5}(y) \res T$ are incompatible.
\end{quote}
Let $T_R$ witness that $R$ is $\phi_5$-recognizable, and set $T = A^{-1} \cup A^{-1} T_R$. Fix $x, y \in X$ satisfying the stated assumptions. If there is $a \in A$ such that $R$ contains precisely one of $a \cdot x$ and $a \cdot y$ then we are done. So we may suppose that for every $a \in A$, $a \cdot x \in R$ iff $a \cdot y \in R$. Set $A_R = \{a \in A \csuchthat a \cdot x \in R\}$. Since $c_{\phi_5}(x) = c_\phi(x) + 2$, we have that $x \in B \cdot D$ and thus $A \cdot x \cap A \Lambda \cdot \Delta = \varnothing$ by property (d). So it follows from the definition of $\phi_4$ that $A \cdot x \subseteq \dom(\phi_5) \cup R$. Therefore
$$c_{\phi_5}(x) = |\{a \in A \setminus A_r \csuchthat \phi_5(a \cdot x) = 1\}|$$
while
$$c_{\phi_5}(y) = |\{a \in A \setminus A_r \csuchthat a \cdot y \in \dom(\phi_5) \text{ and } \phi_5(a \cdot y) = 1\}|.$$
If $(A \setminus A_R) \cdot y \subseteq \dom(\phi_5)$ then we are done since $c_{\phi_5}(y) < c_{\phi_5}(x)$. On the other hand, if $(A \setminus A_r) \cdot y \not\subseteq \dom(\phi_5)$ then by (\ref{eqn:line1}) we must have that $\phi_5 \res A \cdot y = \phi_4 \res A \cdot y$. Furthermore by (\ref{eqn:line2}) we have $c_{\phi_5}(y) = c_{\phi_4}(y) = c_\phi(y)$ and hence
$$c_{\phi_5}(y) = c_\phi(y) \leq c_\phi(x) = c_{\phi_5}(x) - 2.$$
Now $c_{\phi_5}(y) \leq c_{\phi_5}(x) - 2$ and (\ref{eqn:line1}) together imply that there is $a \in A \setminus A_R$ with $a \cdot y \in \dom(\phi_5)$ and $\phi_5(a \cdot y) \neq \phi_5(a \cdot x)$. This completes the proof of $(\star )$.

We can now use induction on $0 \leq m \leq N$ and the spacing conditions used in the definition of the $D_m$'s to show that each $\Delta_m$ is $\phi_5$-recognizable. We begin with $\Delta_0$. Observe that the set $\{ x\in X \csuchthat c_{\phi _5}(x) = N+2\}$ is $\phi _5$-recognizable by $(\star )$. 
To show that $\Delta _0$ is $\phi _5$-recognizable it therefore suffices to show that for $x\in X$, $x \in \Delta_0$ if and only if
\begin{quote}
$c_{\phi _5}(x) = N + 2$ and for all $y \in B^2 \cdot x$, if $c_{\phi _5}(y) = N + 2$ then $L_{\phi _5}(y) \preceq L_{\phi _5}(x)$.
\end{quote}
First suppose that $x = \delta \in \Delta_0$. Then clearly $c_{\phi _5}(\delta) = N + 2$. Let $d \in D_0$ be such that $\delta = p(d) \in B \cdot d$. Then $B^2 \cdot \delta \subseteq B^3 \cdot d$. By construction we have that $B^3 \subseteq F$ and thus $B^3 \cdot d \cap B^3 \cdot d' = \varnothing$ for all $d' \in D$ with $d' \neq d$. Since every $y \in X$ with $c_{\phi _5}(y) = N + 2$ must lie in $B \cdot D_0$, any such $y$ in $B^2 \cdot \delta$ must lie in $B \cdot d$. Hence, the definition of $p$ gives $L_{\phi _5}(y) \preceq L_{\phi _5}(x)$ for all such $y$. Conversely, suppose that $x$ satisfies the stated condition. Then $c_{\phi _5}(x) = N + 2$ implies that $x \in B \cdot D_0$. Say $x \in B \cdot d$ with $d \in D_0$, and set $\delta = p(d) \in \Delta_0$. We have $x \in B \cdot d$ and thus the definition of $p$ gives $L_{\phi _5}(x) \preceq L_{\phi _5}(\delta )$. On the other hand, $\delta \in B \cdot d \subseteq B^2 \cdot x$ and $c_{\phi _5}(\delta) = N + 2$. So the assumption on $x$ implies that $L_{\phi _5}(\delta ) \preceq L_{\phi _5}(x)$. Now $x, \delta \in B \cdot d$ and the construction of $\phi_3$ guarantees that $L_{\phi _5}(z) \neq L_{\phi _5}(z')$ for $z \neq z' \in B \cdot d$, so we conclude $x = \delta \in \Delta_0$.

Now for the inductive step fix $0 < m \leq N$ and assume that $\Delta_t$ is $\phi _5$-recognizable for all $0\leq t < m$. Then $D_t$ is also $\phi_5$-recognizable for all $0\leq t < m$. Fix $x=\delta \in \Delta _m$ and $y\not\in \Delta _m$. Let $b \in B$ be such that $b \cdot \delta  =d \in D_m$, where $p(d) = \delta$. We note the following:
$$c_{\phi _5}(\delta ) = N - m + 2;$$
$$c_{\phi _5}(b \cdot \delta ) = N - m + 2;$$
$$b \cdot \delta \not\in \bigcup_{0 \leq t < m} B^{3m+1} F^{-1} F B \cdot D_t;$$
$$\text{and} \quad L_{\phi _5}(\delta ) \succeq L_{\phi _5}(z) \text{ for all } z \in B^2 \cdot \delta \text{ with } c_{\phi _5}(z) = N - m + 2.$$
The first and second lines follow from the construction. The definition of $D_m$ implies that $b \cdot \delta = d$ satisfies the condition in the third line. By property (3) we have that $c_\phi(z) \leq N - m$ for all $z \in B^2 \cdot \delta \subseteq B^3 \cdot d$. Therefore every $z \in B^2 \cdot \delta$ with $c_{\phi _5}(z) = N - m + 2$ must lie in $B \cdot D$, and since $B^3 \cdot d \cap B \cdot d' = \varnothing$ for $d \neq d' \in D$, each such $z$ must lie in $B \cdot d$. It follows from the definition of $p$ that $L_{\phi _5}(\delta) \succeq L_{\phi _5}(z)$ for all such $z$.

We will now consider finitely many cases, and show that in each case there is a finite subset of $G$ on which $\wt{\phi}_5(x)$ and $\wt{\phi}_5(y)$ are incompatible. If $b\cdot y\in \bigcup_{0 \leq t < m} B^{3m+1} F^{-1} F B \cdot D_t$ then we are done, since this set is $\phi _5$-recognizable by the induction hypothesis. So assume
\begin{equation}\label{eqn:assumption1}
b\cdot y \not\in \bigcup _{0\leq t <m}B^{3m+1} F^{-1} F B \cdot D_t.
\end{equation}
Then it follows from property (1) that for every $i<m$, $c_{\phi}(b\cdot y)\neq N-i$. So $c_\phi (b\cdot y)\leq N-m$. If $c_{\phi _5}(b\cdot y)\leq N-m+1$ then we are done by $(\star )$. So we may additionally assume that
\begin{equation}\label{eqn:assumption2}
c_{\phi _5}(b\cdot y)= N-m+2 .
\end{equation}
This implies $c_\phi (b\cdot y)=N-m$ and $b\cdot y\in B\cdot D_j$ for some $0\leq j\leq N$. Property (3) then implies that $N-m=c_{\phi}(b\cdot y)\leq N-j$, from which we obtain $j\leq m$, and (\ref{eqn:assumption1}) implies that $b\cdot y\not\in B\cdot D_t$ for any $t<m$. Thus $j=m$ and we may find some $d'\in D_m$ with $b\cdot y\in B\cdot d$. Then $y\in B^2\cdot d'$ and property (3) implies that $c_\phi (y) \leq N-m$. The case where $c_{\phi _5}(y)\leq N-m+1$ is again handled by $(\star )$, so we can assume from now on that
\begin{equation}\label{eqn:assumption3}
c_{\phi _5}(y)= N-m+2.
\end{equation}
Then $y\in B\cdot D$ and so $y\in B\cdot d'$. Set $\delta ' = p(d')\in \Delta _m$ and fix $b'\in B^2$ with $b'\cdot y = \delta ' \in \Delta _m$. By assumption, $y\not\in \Delta _m$, so we must have $L_{\phi _5}(b'\cdot y )\succneqq L_{\phi _5}(y)$. If $c_{\phi _5}(b'\cdot \delta ) = N-m+2$ then $L_{\phi _5}(b'\cdot \delta )\preceq L_{\phi _5}(x)$ by the properties we established for $\delta$ above, so we are done. The final possibility is that $c_{\phi _5}(b'\cdot \delta )\neq N-m+2$, in which case $c_{\phi _5}(b'\cdot \delta )\leq N-m+1$ and $c_{\phi}(b'\cdot \delta )\leq N-m$ by property (3), so we are done by $(\star )$ once again. This completes the proof that $\Delta$ is $\phi_5$-recognizable.

Now that we have constructed a syndetic recognizable set $\Delta$, the remainder of the proof becomes much simpler. We now define $R'$ and we will soon define $M'$ and $\phi'$. Define
$$R' = M \cap (X \setminus \dom(\phi_5)) \cap A \cdot \lambda_\ell \cdot \Delta.$$
Recall from the definitions of $\phi_4$ and $\phi_5$ that for every $\delta \in \Delta$ there is precisely one $a \in A$ with $a \lambda_\ell \cdot \delta \in M \cap (X \setminus \dom(\phi_5))$. It is clear from the definition that $R'$ is Borel, but $R'$ might not be $\phi_5$-recognizable. This is easily fixed. Let $\phi_6$ be the extension of $\phi_5$ with
$$\dom(\phi_6) = \dom(\phi_5) \cup \Big(M \bigcap A \cdot \{\lambda_{K+1}, \lambda_{K+2}, \ldots, \lambda_{2K}\} \cdot \Delta \Big)$$
and satisfying for each $\delta \in \Delta$ and $1 \leq i \leq K$
$$c_{\phi_6}(\lambda_{K+i} \cdot \delta) \equiv \bB_i(r(a)) \mod 2,$$
where $a \in A$ is such that $a \cdot \lambda_\ell \cdot \delta \in R'$. Then it is not difficult to see that $R'$ is $\phi_6$-recognizable, contained in $M$, and is syndetic since $D$ is syndetic (property (2) above) and $D \subseteq B \lambda_\ell^{-1} A^{-1} \cdot R'$.

Lastly, we define $\phi'$ and $M'$. Define
$$M' = M \cap (X \setminus \dom(\phi_6)) \cap A \cdot \lambda_{\ell-1} \cdot \Delta.$$
Then $M' \subseteq M$ is disjoint from $R'$ and is syndetic since $D \subseteq B \lambda_{\ell-1}^{-1} A^{-1} \cdot M'$. Let $1_G \neq s \in G$ be the group element from the statement of the proposition. Let $\mc{G}$ be the Borel graph with vertex set $\Delta$ and edge relation
$$(\delta, \delta') \in \mc{G} \Longleftrightarrow \exists h \in B^2 F F^{-1} B^{3N+1} \ \Big( \delta = h s h^{-1} \cdot \delta' \mbox{ or } \delta' = h s h^{-1} \cdot \delta \Big).$$
Each vertex of $\mc{G}$ has degree at most $2 |B|^{3N+3} |F|^2$, so we can apply Lemma \ref{lem:color} to obtain a proper vertex coloring $\kappa : \Delta \rightarrow \{0, 1, \ldots, 2 |B|^{3N+3} |F|^2\}$ of $\mc{G}$. We let $\phi' : X \setminus (M' \cup R') \rightarrow \{0, 1\}$ be the extension of $\phi_6$ which satisfies for every $\delta \in \Delta$ and $1 \leq i \leq \ell - 2K - 2$
$$c_{\phi'}(\lambda_{2K+i} \cdot \delta) \equiv \bB_i(\kappa (\delta)) \mod 2.$$
Fix $T_R\subseteq G$ finite witnessing that $R$ is $\phi$-recognizable. Since
$$\ell - 2K - 2 \geq \log_2(2 |B|^{3N+3} |F|^2 + 1)$$
we have that if $\delta, \delta' \in \Delta$, and $\kappa (\delta) \neq \kappa (\delta')$ then there is $1 \leq i \leq \ell - 2K - 2$ and $a \in A$ such that either
$$\chi_R(a \lambda_{2K+i} \cdot \delta) \neq \chi_R(a \lambda_{2K+i} \cdot \delta'),$$
where $\chi_R$ is the characteristic function of $R$, or else
$$\phi'(a \lambda_{2K+i} \cdot \delta) \neq \phi'(a \lambda_{2K+i} \cdot \delta').$$
It follows in either case that there is some $g\in T_R^{-1}A\Lambda \cup A\Lambda$ with $g\cdot \delta, g \cdot \delta' \in \dom(\phi')$ and $\phi'(g \cdot \delta) \neq \phi'(g \cdot \delta')$.

Fix $T_\Delta$ witnessing that $\Delta$ is $\phi '$-recognizable and let
$$T= (T_\Delta ^{-1}\cup T_R^{-1}A\Lambda \cup A\Lambda )B^2FF^{-1}B^{3N+1}.$$
It remains to show that for every $x \in X$ there is $t \in T$ with $t \cdot x, t s \cdot x \in \dom(\phi')$ and $\phi'(t \cdot x) \neq \phi'(t s \cdot x)$, which will prove part (iv). Fix $x \in X$. By property (2) and the containment $D \subseteq B \cdot \Delta$, there is $h \in B^2 F F^{-1} B^{3N+1}$ with $h \cdot x = \delta \in \Delta$. If $h s \cdot x \not\in \Delta$ then there is some $g \in T_\Delta$ with $\phi ' (g ^{-1}h \cdot x) \neq \phi ' (g^{-1} h s \cdot x)$. So we are done if $h s \cdot x \not\in \Delta$. Now suppose that $h s \cdot x = \delta' \in \Delta$. Then
$$\delta' = h s \cdot x = h s h^{-1} \cdot h \cdot x = h s h^{-1} \cdot \delta.$$
Thus $\delta$ and $\delta'$ are joined by an edge in $\mc{G}$ and
$$\kappa (h \cdot x) = \kappa (\delta) \neq \kappa (\delta') = \kappa (h s \cdot x).$$
It follows from the last remark of the previous paragraph that there is $g \in T_R^{-1}A\Lambda \cup A\Lambda$ with $\phi ' (g h \cdot x) \neq \phi ' (g h s \cdot x)$.
\end{proof}

\begin{proof}[Proof of Theorem \ref{thm:main}.(2)]
Let $G \cc X$ be a free Borel action, let $Y \subseteq X$ be Borel with $X \setminus Y$ syndetic, and let $\phi : Y \rightarrow \{0, 1\}$ be a Borel function. Fix an enumeration $s_1, s_2, \ldots$ of the non-identity elements of $G$. Set $R_0 = \varnothing$, $M_0 = X \setminus Y$, and $\phi_0 = \phi$. We first build a sequence $(\phi_n)_{n \geq 1}$ of Borel functions and sequences $(R_n)_{n \geq 1}$ and $(M_n)_{n \geq 1}$ of syndetic Borel sets. Note that $\phi_0$, $R_0$, and $M_0$ are already defined (although $R_0$ is not syndetic). In general, once $\phi_n$, $R_n$, and $M_n$ have been defined, apply Lemma \ref{IND STEP} using $s = s_{n+1}$, $\phi_n$, $M_n$, and $R_1 \cup \cdots \cup R_n$ to obtain $\phi_{n+1}$, $M_{n+1}$, and $R_{n+1}$. This defines the sequences $(\phi_n)$, $(R_n)$, and $(M_n)$. We have that the $R_n$'s are pairwise disjoint, the $M_n$'s are decreasing, and the $\phi_n$'s are increasing. Define
$$ \phi _\infty : X \setminus \Big(\bigcup_{n \in \N} R_n \Big) \rightarrow \{0, 1\}$$
by setting $\phi _\infty (x) = \phi_n(x)$ if there is $n$ with $x \in \dom(\phi_n)$, and setting $\phi _\infty (x) = 0$ for $x \in \bigcap_{n \in \N} M_n$. Then $\phi _\infty$ is Borel. For $w \in 2^\N$ we extend $\phi _\infty$ to $\phi _w : X \rightarrow \{0, 1\}$ by setting $\phi _w(x) = w(n - 1)$ for $x \in R_n$, $n \geq 1$. Now define $f_w :X \ra 2^G$ by
$$f_w = \wt{\phi}_w ,$$
and let $f:2^\N \times X \ra 2^G$ be the map $f(w,x)=f_w(x)$. Then $f$ is Borel, and for each $x\in X$ the map $f^x: 2^\N \ra 2^G$ is continuous since if $\lim _{i\ra \infty } w_i = w$ in $2^\N$ then for each $g\in G$ we have
\begin{align*}
\lim _{i\ra\infty} f^x(w_i)(g) = \lim _{i\ra\infty}\phi _{w_i}(g^{-1}\cdot x) &=
\begin{cases}
\lim _{i\ra\infty}w_i(n-1)= w(n-1) &\text{if } g^{-1}\cdot x \in R_n \\
\phi _\infty (g^{-1}\cdot x) &\text{if }g^{-1}\cdot x \not\in \bigcup _{n\in \N}R_n
\end{cases}
\\
&=\phi _w (g^{-1}\cdot x) = f^x(w)(g) .
\end{align*}

We fix $w\in 2^\N$ and check that $\ol{f_w(X)}\subseteq \mathrm{Free}(2^G )$. By clause (iv) of Lemma \ref{IND STEP}, for each $n\geq 1$ there is a finite set $T_n\subseteq G$ such that for all $x\in X$ the functions $\wt{\phi} _n(x)\res T_n$ and $\wt{\phi}_n(s_n\cdot x)\res T_n$ are incompatible. Therefore, for each $x\in X$, since $f_w(x)$ and $f_w(s_n\cdot x )=s_n\cdot f_w(x)$ are total functions on $G$ extending $\wt{\phi} _n(x)$ and $\wt{\phi }_n(s_n\cdot x )$ respectively, we have that $f_w(x)\res T_n \neq (s_n\cdot f_w(x))\res T_n$. Therefore, $f_w(X)$ is contained in the closed set $\{ u\in 2^G \csuchthat u\res T_n \neq s_n\cdot u \res T_n \text{ for all }n\geq 1 \}$, which in turn is contained in $\mathrm{Free}(2^G)$.

Fix $w \neq z \in 2^\N$. We must check that the closure of the images of $f_w$ and $f_z$ are disjoint. Let $n \geq 1$ be such that $w(n-1) \neq z(n-1)$. Since $R_n \subseteq X$ is syndetic there is a finite subset $S\subseteq G$ with $S\cdot R_n = X$. Let $T\subseteq G$ be a finite set which witnesses that $R_n$ is $\phi _n$-recognizable. It suffices to show that $f_w(x)\res (S\cup ST ) \neq f_z(y)\res (S\cup ST )$ for all $x,y\in X$, since then it follows that $\{ u\in 2^G \csuchthat (\exists x\in X )(u\res (S\cup ST ) = f_w(x) \res (S\cup ST) ) \}$ and $\{ u\in 2^G \csuchthat (\exists x\in X )(u\res (S\cup ST ) = f_z(x) \res (S\cup ST) ) \}$ are disjoint clopen sets containing $f_w(X)$ and $f_z(X)$ respectively. Given $x,y \in X$, by our choice of $S$ we may find $s\in S$ such that $s^{-1}\cdot x \in R_n$. If $s^{-1}\cdot y\in R_n$ then we have
\[
f_w(x)(s) = \phi _w (s^{-1}\cdot x ) = w(n-1) \neq z(n-1)= \phi _z (s^{-1}\cdot y) = f_z(y)(s) ,
\]
so we are done. We may therefore assume that $s^{-1}\cdot y\not\in R_n$. Since $s^{-1}\cdot x\in R_n$ and $s^{-1}\cdot y\not\in R_n$, by our choice of $T$ we may find some $t\in T$ with $t^{-1}s^{-1}\cdot x , t^{-1}s^{-1} \cdot y \in \dom(\phi_n)$ and $\phi_n(t^{-1}s^{-1}\cdot x ) \neq \phi_n(t^{-1} s^{-1} \cdot y)$. Thus
$$f_w(x)(st) = \phi _w(t^{-1}s^{-1} \cdot x) = \phi_n(t^{-1}s^{-1}\cdot x) \neq \phi_n(t^{-1}s^{-1} \cdot y) = \phi_z(t^{-1}s^{-1}\cdot y) = f_z(y)(st ), $$
which finishes the proof.
\end{proof}

\section{Genericity of maps into the free part} \label{sec:generic}

In this section we deduce Theorem \ref{thm:main}.(3) from Theorem \ref{thm:main}.(1). We will need the following lemma. In what follows, if $G$ acts on a set $Y$ then for $g\in G$ let $\mathrm{Fix}_Y(g) = \{ y\in Y \csuchthat g\cdot y =y \}$.

\begin{lemma}\label{lem:char} Let $G \cc Y$ be a Borel action of $G$, let $\nu$ be a Borel probability measure on $Y$, and let $\mc{P}$ be a countable generating partition for $G\cc (Y,\nu )$. Then for any $g\in G$ we have
\begin{equation}
\label{eqn:char}\nu (\mathrm{Fix}_Y(g)) = \inf \left\{ \sum _{P\in \mc{P}^Q}\nu (g\cdot P \cap P ) \csuchthat Q\subseteq G\text{ is finite} \right\} ,
\end{equation}
where $\mc{P}^Q = \bigvee _{h \in Q}h\cdot \mc{P}$.
\end{lemma}

\begin{proof}
For any $g\in G$ and $P\subseteq X$ we have $\mathrm{Fix}_Y(g) \cap P \subseteq g\cdot P\cap P$. Therefore, for any $Q\subseteq G$ finite we have $\nu (\mathrm{Fix}_Y(g)) = \sum _{P\in \mc{P}^Q}\nu ( \mathrm{Fix}_Y(g)\cap P) \leq \sum _{P\in \mc{P}^Q} \nu (g\cdot P\cap P)$, and taking the infimum over $Q$ proves the inequality $\leq$ for (\ref{eqn:char}). For the other inequality, given $g\in G$, apply Lemma \ref{lem:color} to the Borel graph $\{ (y,s\cdot y) \csuchthat y\in Y\setminus \mathrm{Fix}_Y(g), \ s\in \{ g, g^{-1}\} \}$ to obtain a Borel partition $\{ A_0,A_1,A_2 \}$ of $Y\setminus \mathrm{Fix}_Y(g)$ with $g\cdot A_i \cap A_i =\emptyset$ for each $i\in \{ 0,1,2 \}$. Let $A_3= \mathrm{Fix}_Y(g)$. Since $\mc{P}$ is generating, for any $\epsilon >0$ we may find a finite $Q\subseteq G$ along with a coarsening $\{ B_0, \dots , B_3 \}$ of $\mc{P}^Q$ such that $\sum _{i\leq 3}\nu (B_i\triangle A_i ) < \epsilon /2$. Then
\begin{align*}
\textstyle{\sum _{P\in \mc{P}^Q} \nu (g\cdot P\cap P ) }&\textstyle{\leq \sum _{i\leq 3}\nu (g\cdot B_i \cap B_i ) \leq \sum _{i\leq 3}\big( \nu (g\cdot A_i \cap A_i ) + 2\nu (B_i\triangle A_i )\big) }\\
&\textstyle{\leq \nu (\mathrm{Fix}_Y(g)) + 2\sum _{i\leq 3}\nu (B_i\triangle A_i ) < \nu (\mbox{Fix}_Y(g) )+ \epsilon .} \qedhere
\end{align*}
\end{proof}

\begin{proposition}\label{prop:Gdelta}
Let $G \cc X$ be a Borel action of the countable group $G$ on a standard Borel space $X$ and let $\mu$ be a $G$-quasi-invariant Borel probability measure on $X$. Then the set
\[
\Es{B}_\mu = \{ [A]_\mu \csuchthat A\subseteq X\text{ is Borel and }f_A\text{ is class-bijective on some }G\text{-invariant conull set}\}
\]
is $G_\delta$ in $\mathrm{MALG}_\mu$.
\end{proposition}

Note that if $G\cc X$ is free then being class-bijective is equivalent to $f_A(X) \subseteq \mathrm{Free}(2^G)$ and thus the set $\Es{B}_\mu$ coincides with the set from Theorem \ref{thm:main}.(3). Specifically, if $G \acts X$ is free and $f_A : X \rightarrow 2^G$ fails to be class-bijective on an invariant null set $Z \subseteq X$, then we can apply Theorem \ref{thm:main} to $G \acts Z$ to get an equivariant class-bijcetive map $f_0 : Z \rightarrow \mathrm{Free}(2^G)$. Now the function $f_0 \cup (f_A \res (X \setminus Z))$ is equivariant and class-bijective and is of the form $f_B$ where $[B]_\mu = [A]_\mu$.

\begin{proof}
Note that if $[A]_\mu = [B]_\mu$ then $f_A$ and $f_B$ agree on a conull subset of $X$. For every Borel $A\subseteq X$, since $f_A$ is $G$-equivariant we have $\mathrm{Fix}_X(g)\subseteq f_A^{-1}(\mathrm{Fix}_{2^G}(g))$ for all $g\in G$. We now claim that
\begin{equation}\label{eqn:Bmu}
[A]_\mu \in \Es{B}_\mu \IFF \forall g\in G,  \ ((f_A)_*\mu )(\mathrm{Fix}_{2^G}(g) ) \leq \mu ( \mathrm{Fix}_X(g) ) .
\end{equation}
Indeed, if $[A]_\mu \in \Es{B}_\mu$ then there is an invariant conull $X_0\subseteq X$ such that $f_A\res X_0$ is class-bijective, hence for every $g\in G$ we have $X_0 \cap f_A^{-1}(\mathrm{Fix}_{2^G}(g)) = \{ x \in X_0\csuchthat f_A(g\cdot x) = f_A(x) \} = \{ x\in X_0\csuchthat g \cdot x = x \} = X_0\cap \mathrm{Fix}_X(g)$. Conversely, suppose that the condition on the right side holds. Then for each $g\in G$ we have $\mu ( \{ x \in X\csuchthat g\cdot x \neq x \text{ and }f_A(g\cdot x )= f_A(x) \} ) = 0$, and since $G$ is countable this implies $\mu ( \{ x \in X\csuchthat (\exists g \in G )( g \cdot x \neq x \text{ and } f_A (g\cdot x ) = f_A (x)  ) \} ) = 0$. If we let $X_0$ denote the complement of this last set, then $X_0$ is $G$-invariant, conull, and $f_A\res X_0$ is class-bijective.

Let $\mc{P}$ denote the canonical generating partition for $G\cc 2^G$, i.e., $\mc{P} = \{ C_0, C_1 \}$, where $C_0= \{ w\in 2^G\csuchthat w(1_G)= 0 \}$ and $C_1 = \{ w\in 2^G\csuchthat w(1_G)=1 \}$. Given $Q\subseteq G$ finite and $\sigma \in 2^Q$ let $C_\sigma = \{ w\in 2^G \csuchthat w\text{ extends }\sigma \}$, and for $A\subseteq X$ Borel let $A_\sigma = \bigcap _{h\in Q}h\cdot A_{\sigma (h)}$, where $A_0 = X\setminus A$ and $A_1 = A$. Then $f _A ^{-1} (C_\sigma ) = A_\sigma$ and $\mc{P}^Q = \{ C_\sigma \} _{\sigma \in 2^Q}$. Therefore, by equation (\ref{eqn:Bmu}) and Lemma \ref{lem:char}, we have $[A]_\mu \in \Es{B}_\mu$ if and only if
\[
\forall g\in G ,\ \inf \left\{ \sum _{\sigma \in 2^Q} \mu (g\cdot A_\sigma \cap A_\sigma ) \csuchthat Q\subseteq G\text{ is finite}\right\} \leq \mu (\mbox{Fix}_X (g)).
\]
This shows that
\[
\Es{B}_\mu = \bigcap _{g\in G} \bigcap _{\epsilon > 0}\bigcup _{\substack{Q\subseteq G \\ \text{finite}}} \Big\{ [A]_\mu \in \mathrm{MALG}_\mu \csuchthat \sum _{\sigma \in 2^Q} \mu (g\cdot A_\sigma \cap A_\sigma ) < \mu (\mbox{Fix}_X (g )) + \epsilon  \Big\} .
\]
To prove that $\mc{B}_\mu$ is $G_\delta$ it therefore remains to show that for any $g\in G$, $Q\subseteq G$ finite and $r\in \R$, the set $\{ [A]_\mu \in \mathrm{MALG}_\mu \csuchthat \sum _{\sigma \in 2^Q} \mu (g\cdot A_\sigma \cap A_\sigma ) < r \}$ is open in $\mathrm{MALG}_\mu$. This follows from the fact that the action $G\cc \mathrm{MALG}_\mu$ is continuous (since $\mu$ is $G$-quasi-invariant), and the Boolean operations are continuous on $\mathrm{MALG}_\mu$, hence the map $[A]_\mu\mapsto \sum _{\sigma \in 2^Q}\mu (g\cdot A_\sigma \cap A_\sigma )$ is continuous on $\mathrm{MALG}_\mu$.
\qedhere[Proposition \ref{prop:Gdelta}]
\end{proof}

\begin{proof}[Proof of Theorem \ref{thm:main}.(3)]
Let $\mu$ be a $G$-quasi-invariant Borel probability measure on $X$. We must show that the set
\[
\{ [A]_\mu \csuchthat A\subseteq X\text{ is Borel and }f_A(X)\subseteq \mathrm{Free}(2^G)\}
\]
is dense $G_\delta$ in $\mathrm{MALG}_\mu$. It is $G_\delta$ by Proposition \ref{prop:Gdelta}, so it remains to show it is dense. We will in fact show that the set $\{ [A]_\mu \csuchthat A\subseteq X\text{ is Borel and }\ol{f_A(X)}\subseteq \mathrm{Free}(2^G) \}$ is dense in $\mathrm{MALG}_\mu$, and we note that the argument does not use quasi-invariance of $\mu$. Fix a Borel subset $B \subseteq X$ and $\epsilon >0$. By Proposition \ref{prop:synd} there exists a syndetic Borel subset $M\subseteq X$ with $\mu (M) < \epsilon$. Let $Y= X\setminus M$ and let $\phi : Y \ra 2$ be given by $\phi (y)=1_B(y)$ for $y\in Y$. By Theorem \ref{thm:main}.(1) there exists a $G$-equivariant Borel map $f: X\ra 2^G$ with $\ol{f(X)}\subseteq \mathrm{Free}(2^G)$ with $f(y)(1_G)= \phi (y) = 1_B(y)$ for all $y\in Y$. Let $A= \{ x\in X\csuchthat f(x)(1_G) = 1 \}$. Then $A$ is Borel, $f_A = f$, and $A\triangle B \subseteq X\setminus Y = M$, hence $\mu (A\triangle B ) < \epsilon$.
\end{proof}

\subsection*{Acknowledgments}\addcontentsline{toc}{section}{Acknowledgements}
This research was initiated while the authors attended a descriptive set theory workshop at the Erwin Schr\"{o}dinger Institute in Vienna and the Arbeitsgemeinschaft: Sofic Entropy workshop at the Mathematisches Forschungsinstitut Oberwolfach in Oberwolfach. The authors thank the ESI and the MFO for their kind hospitality and travel support. The authors would also like to thank Andrew Marks for asking the question which sparked this collaboration. The first author was supported by the National Science Foundation Graduate Student Research Fellowship under grant DGE 0718128. The second author was supported by NSF grant DMS 1303921.


\bibliographystyle{amsalpha}
\bibliography{biblio}

\end{document}